\documentclass[12pt]{article}
\usepackage{geometry}
\usepackage{float}
\usepackage{latexsym,bm}
\usepackage{amsmath,amsfonts,amsthm,amssymb,bbding,mathrsfs}
\usepackage{graphicx}
\usepackage{hyperref}
\usepackage{xcolor}

\newtheorem{thm}{Theorem}[section]

\newtheorem{lemma}{Lemma}[section]
\newtheorem{cor}{Corollary}[section]
\newtheorem{pro}{Property}[section]

\newtheorem{rem}{Remark}[section]

\newtheorem{definition}{Definition}[section]
\newtheorem{case}{Case}[section]

\newcommand{\ex}{\mathrm{ex}}

\title{Planar Tur\'an number of two adjacent cycles}
\author{Xinzhe Song\thanks{Academy of Mathematics and Systems Science, Chinese Academy of Sciences, Beijing, China, and University of Chinese Academy of Sciences, Beijing, China.}
\and Guiying Yan\footnotemark[1]
\and Qiang Zhou\footnotemark[1]
}
\date{}

\begin{document}

\maketitle

\begin{abstract}
The planar Tur\'an number of $H$, denoted by $\ex_{\mathcal{P}}(n,H)$, is the maximum number of edges in an $n$-vertex $H$-free planar graph. The planar Tur\'an number of $k(k\geq 3)$ vertex-disjoint union of cycles is the trivial value $3n-6$. We determine the planar Tur\'an number of $C_{3}\text{-}C_{3}$ and $C_{3}\text{-}C_{4}$, where $C_{k}\text{-}C_{\ell}$ denotes the graph consisting of two disjoint cycles $C_k$ with an edge connecting them.


\end{abstract}

\maketitle
\section{Introduction}

One of the most classical problems in extremal graph theory called Tur\'an-type problem, is counting the maximum number of edges in a hereditary family of graphs. We use $\ex(n,H)$ to denote the maximum number of edges in an $n$-vertex $H$-free graph. In 1941, Tur\'an~\cite{turan} gave the exact value of $\ex(n,K_{r})$ and characterized the extremal graph. Later in 1946, Erd\H{o}s and Stone~\cite{erdos1946} generalized this result asymtotically by proving $\ex(n,H)=(1-\frac{1}{\chi(H)-1}+o(1))\binom{n}{2}$ for an arbitrary graph $H$, where $\chi(H)$ denotes the chromatic number of $H$.

All graphs considered in this paper are simple and planar. We denote a simple graph by $G=(V(G),E(G))$ where $V(G)$ is the set of vertices and $E(G)$ is the set of edges. Let $v(G)$ and $e(G)$ denote the number of vertices and edges in $G$, respectively. For any subset $S\subseteq V(G)$, the subgraph induced by $S$ is denoted by $G[S]$. We denote by $G\backslash S$ the subgraph induced by $V(G)\backslash S$. If $S=\{v\}$, then we simply write $G\backslash v$. Given two graphs $G$ and $H$, $G$ is $H$-\emph{free} if it contains no subgraph isomorphic to $H$. We use $G \vee H$ to denote the graph obtained from $G$ and $H$ by joining all edges between $V(G)$ and $V(H)$. Given two subgraphs $H_1$ and $H_2$ in a planar graph $G$, we use $H_1\text{-}H_2$ to denote two vertex-disjoint graphs $H_1$ and $H_2$ with an edge connecting them. For more undefined notation, please refer to~\cite{graph-theory}.

In 2016, Dowden~\cite{dowden2016} initiated the study of Tur\'an-type problems when host graphs are planar graphs. We use $\ex_{\mathcal{P}}(n,H)$ to denote the maximum number of edges in an $n$-vertex $H$-free planar graph. Dowden studied the planar Tur\'an numbers of $C_{4}$ and $C_{5}$, they proved that $\ex_{\mathcal{P}}(n,C_4) \leq \frac{15(n-2)}{7}$ and $\ex_{\mathcal{P}}(n,C_5) \leq \frac{12n-33}{5}$, where $C_{k}$ is a cycle with $k$ vertices. Ghosh, Gy\H{o}ri, Martin, Paulos and Xiao~\cite{ghosh2022c6} gave the exact value for $C_{6}$, they proved that $\ex_{\mathcal{P}}(n,C_6) \leq \frac{5n-14}{2}$. Shi, Walsh and Yu~\cite{shi2025planar}, Gy\H{o}ri, Li and Zhou~\cite{győri2023c7} gave the exact value for $C_{7}$, they proved that $\ex_{\mathcal{P}}(n,C_7) \leq \frac{18n-48}{7}$, independently. Until now, the exact planar Tur\'an number of $C_{k}$ is still unknown for $k\geq 8$.

Cranston, Lidick\'{y}, Liu and Shantanam~\cite{daniel2022counterexample} first gave the lower bound for general cycles as $\ex_{\mathcal{P}}(n,C_k) > \frac{3(k-1)}{k}n-\frac{6(k+1)}{k}$ for $k\geq 11$ and sufficiently large $n$. Lan and Song~\cite{lan2022improved} improved the lower bound to $\ex_{\mathcal{P}}(n,C_{k}) > (3-\frac{3-\frac{2}{k-1}}{k-6+\lfloor \frac{k-1}{2}\rfloor})n+c_{k}$ for $n\geq k\geq 11$, where $c_{k}$ is a constant. Recently, Gy\H{o}ri, Varga and Zhu~\cite{gyHori2024new} improved the lower bound by giving a new construction to show that $\ex_{\mathcal{P}}(n,C_{k}) \geq 3n-6-\frac{6\cdot3^{\log_23}n}{k^{\log_23}}$ for $k\geq 7$ and sufficiently large $n$. Shi, Walsh and Yu~\cite{shi2025dense} gave an upper bound and proved that $\ex_{\mathcal{P}}(n,C_{k}) \leq 3n-6-\frac{n}{4k^{\log_23}}$ for all $n,k\geq4$.

We use $tC_{k}$ to denote the union of $t$ vertex-disjoint $k$-cycles and $t\mathcal{C}$ to denote the union of $t$ vertex-disjoint cycles without length restriction. Lan, Shi and Song~\cite{lan2019hfree} showed that $\ex_{\mathcal{P}}(n,t\mathcal{C})=3n-6$ for $t\geq 3$ and the extremal graph is the double wheel $2K_{1}+C_{n-2}$.  Lan, Shi and Song~\cite{lan2024planar} proved that $\ex_{\mathcal{P}}(n,2C_{3})=\lceil 5n/2 \rceil - 5$ for $n\geq 6$. Li~\cite{li2024} proved that $\ex_{\mathcal{P}}(n,C_{3}\cup C_{4})=\lfloor 5n/2 \rfloor - 4$ for $n\geq 20$, where $C_3\cup C_4$ denotes the union of two vertex-disjoint cycles $C_3$ and $C_4$. Fang, Lin and Shi~\cite{fang2024} proved that for $n\geq 2661$, $\ex_{\mathcal{P}}(n,2C_{4})=19n/7 - 6$ if $7|n$ and $\ex_{\mathcal{P}}(n,2C_{4})=\lfloor (19n-34)/7 \rfloor$ otherwise, they also proved that  $\ex_{\mathcal{P}}(n,2\mathcal{C})=2n-1$, where $\mathcal{C}$ denotes the family of cycles.
In this paper, we determine the exact value of the planar Tur\'an number for $C_3\text{-}C_3$ and $C_3\text{-}C_4$, where $C_{k}\text{-}C_{\ell}$ denotes the graph consisting of two disjoint cycles $C_k$ with an edge connecting them.

\begin{thm}\label{thm1}
    Let $n\geq 3$. Then
    \begin{align*}
        \begin{split}
            \ex_{\mathcal{P}}(n,C_3\text{-}C_3) = \left\{
                \begin{array}{ll}
                3n-6 & \text{if } n\leq 5,\\
                3n-7 & \text{if } n=6,\\
                \lceil 5n/2 \rceil - 5 & \text{if } 
                n\geq 7.
                \end{array}
            \right.
        \end{split}
    \end{align*}
\end{thm}

\begin{thm}\label{thm2}
    Let $n\geq 3$. Then
    \begin{align*}
        \begin{split}
            \ex_{\mathcal{P}}(n,C_3\text{-}C_4) = \left\{
                \begin{array}{ll}
                3n-6 & \text{if } n\leq 6,\\
                3n-7 & \text{if } n=7,\\
                \lfloor 5n/2 \rfloor - 4 & \text{if } 
                n\geq 8.
                \end{array}
            \right.
        \end{split}
    \end{align*}
\end{thm}

Since $C_3 \cup C_k$-free graphs are also $(C_3\text{-}C_k)$-free and our extremal graphs are also $C_3 \cup C_k$-free, we could deduce the planar Tur\'an number of $2C_3$ and $C_3 \cup C_4$ without any restriction to $n$. By the way, there is a tiny error about planar Tur\'an number of $2C_3$ in~\cite{lan2024planar}. They showed that $\ex_{\mathcal{P}}(6,2C_3)=10$ but actually it is $11$ as shown in Figure~\ref{face-blocks-equation}. Here we revise their result as a corollary. 

\begin{cor}\cite{lan2024planar}
    Let $n\geq 3$. Then
    \begin{align*}
        \begin{split}
            \ex_{\mathcal{P}}(n,2C_3) = \left\{
                \begin{array}{ll}
                3n-6 & \text{if } n\leq 5,\\
                3n-7 & \text{if } n=6,\\
                \lceil 5n/2 \rceil - 5 & \text{if } 
                n\geq 7.
                \end{array}
            \right.
        \end{split}
    \end{align*}
\end{cor}

\begin{cor}\cite{li2024}
    Let $n\geq 3$. Then
    \begin{align*}
        \begin{split}
            \ex_{\mathcal{P}}(n,C_3\cup C_4) = \left\{
                \begin{array}{ll}
                3n-6 & \text{if } n \leq 6,\\
                3n-7 & \text{if } n=7,\\
                \lfloor 5n/2 \rfloor - 5 & \text{if } 
                n\geq 8.
                \end{array}
            \right.
        \end{split}
    \end{align*}
\end{cor}

\section{Preliminaries}

We refer to a planar embedding of a planar graph as \emph{plane graph}.
For a plane graph $G$, a face of size $i$ is called an $i$-\emph{face}. The number of $i$-faces and the total number of faces in a plane graph $G$ are denoted by $f_i(G)$ and $f(G)$, respectively. Let $E_{i,j}$ denote the set of edges in $G$ that belong to one $i$-\emph{face} and one $j$-\emph{face} (and belong to two $i$-\emph{faces} when $i=j$), $E_i$ denote the set of edges in $G$ that belong to at least one $i$-\emph{face}. Let $e_{i,j}=|E_{i,j}|$ and $e_i = |E_i|$. Let $e(G_1:G_2)$ denote the edges between  the graph $G_1$ and $G_2$. We will get the following property.

\begin{pro}\label{pro1}  Let $G$ be a plane graph. For all $i \geq 3$,
    \begin{itemize}
        \item[1)] $e_{i,i} \leq e_i \leq e(G)$;
        
        \item[2)] $if_i(G) = e_i+e_{i,i}$.
        
    \end{itemize}
\end{pro}

Let $G\vee H$ denote the graph obtained from $G$ and $H$ by joining all edges between $V(G)$ and $V(H)$.
A \emph{fan} is the graph $K_1\vee H$ where $H$ is a path. A \emph{wheel} is the graph $K_1\vee H$ where $H$ is a cycle. A $\Theta_{k}$ is the family of Theta graphs on $k\geq4$ vertices, that is, graphs obtained from a cycle $C_k$ by adding an additional edge joining two non-consecutive vertices. Given a vertex $v \in V(G)$, let $\overline{R_v}$ denote the set of $3$-faces containing $v$, that is $\overline{R_v}=\{F \ | \ F \ \text{is a 3-face} \ \text{and} \ F \ \text{is incident with } v\}$. Then we have:

\begin{equation*}\label{equation}
    \sum_{v \in V(G)}|\overline{R_v}|=3f_3(G).
\end{equation*}

We regard $R_v$ as a subgraph of $G$, which satisfies $V(R_v)=V(\overline{R_v})$ and $E(R_v)=\{e \ | \ e \in F \ \text{and} \ F \ \text{is incident with } v\}$. For each vertex $v$ with $\overline{R_v} \neq \emptyset$, $R_v$ is a wheel or $R_v$ consists of $\ell$ fans $F_1, F_2,...,F_{\ell}$. 
We use $(k_1,k_2,...,k_{\ell})$ to denote the partition of $R_v$ such that each part $F_i$ contains $k_i$ $3$-faces of $R_v$ for $i\in[\ell]$, where $\sum_{i=1}^{\ell}k_i=| \overline{R_v}|$ and ${k_{1}\geq k_{2}\geq\cdots\geq k_{\ell}\geq 1}$, and we say $F_1$ is the first part, $F_2$ is the second part and so on. For convenience, we label all the vertices incident to $v$ in each $F_i$ by $\{ v_{k_0+\cdots+k_{i-1}+i}, v_{k_0+\cdots+k_{i-1}+i+1},...,v_{k_0+\cdots+k_{i-1}+k_{i}+i} \}$ in a clockwise direction, where $k_0=0$ and $i\in[\ell]$.
See Figure~\ref{Block-partition} for an example when the partition of $R_v$ is $(2,2,1,1)$.

\begin{figure}[H]
    \centering
    \includegraphics[width=0.3\linewidth]{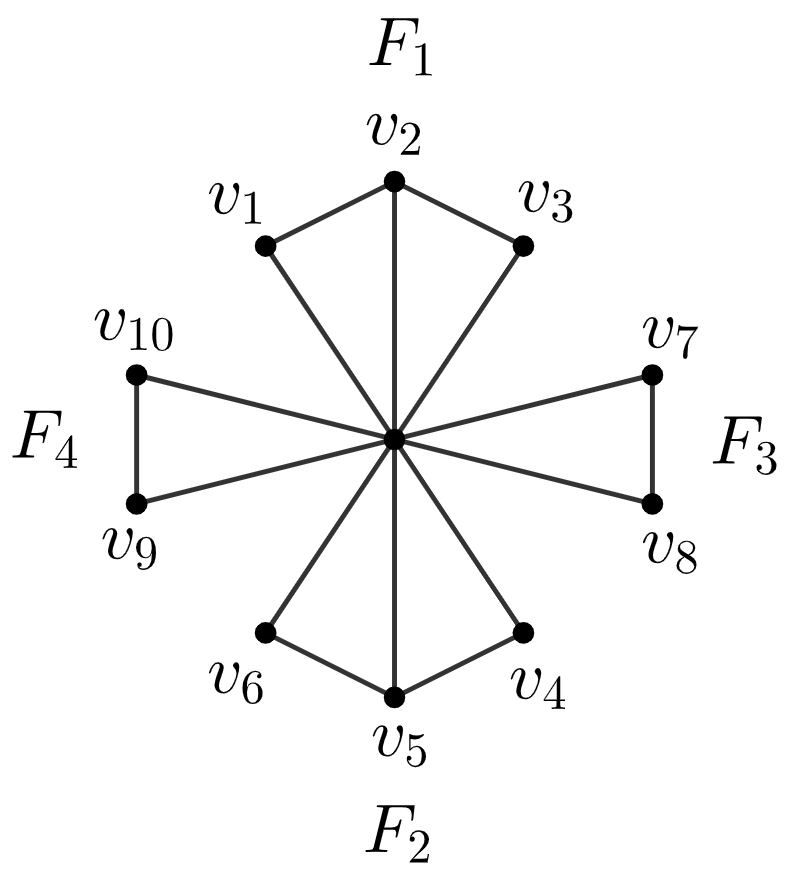}
    \caption{$R_v$ with the partition $(2,2,1,1)$.}
    \label{Block-partition}
\end{figure}

\begin{definition}\label{def}
    Let $G$ be a plane graph, A $3$-face-block $B\subseteq V(G)$ is a maximal connected subgraph which satisfies that for any two vertices in $B$, there is a path connecting them and each edge of this path is contained in some $3$-face.
\end{definition}

Let $G$ be a plane graph. A $3$-face-block $B\subseteq V(G)$ can be built as follows : Begin from a vertex $v\in V(G)$,
\begin{itemize}
    \item[1)]  if $v$ is not incident with any $3$-face of $G$, then we take $B=\{v\}$;
    
    \item[2)] Otherwise, we start from $B_1=\{v\}$, $B_2=\{u | u\ \text{is incident with $v$ by a 3-face}\}$, $B_3=\{u | u\ \text{is incident with the vertex in $B_2$ by a $3$-face and} \ u \notin B_1 \cup B_2\}$ and similarly we define $B_k=\{u | u\ \text{is incident with some vertex in $B_{k-1}$ by a $3$-face and} \ u \notin \bigcup\limits_{i=1}^{k-1} B_i\}$. Then the $3$-face-block built from $v$ is $B=\bigcup\limits_{i=1}^{\infty} B_i$.
    
\end{itemize}

Let $\mathcal{B}(G)$ denote the set of all $3$-face-blocks in a plane graph $G$. Then we have:

\begin{rem}
     $\mathcal{B}(G)$ is unique, and $\mathcal{B}(G)$ is a partition of $V(G)$.
\end{rem}

\section{Proof of Theorem~\ref{thm1}}

To prove Theorem \ref{thm1}, we give some auxiliary lemmas.

\begin{lemma}\label{R-v}
    Given a vertex $v$ in a $(C_3\text{-}C_3)$-free plane graph, if $R_u\subseteq R_v$ for each vertex $u\in V(R_v)$, then $R_v$ is a $3$-face-block and $\sum_{u \in R_v}|\overline{R_u}| \leq 3|V(R_v)|-3$.
\end{lemma}

\begin{proof}
    Since $R_u\subseteq R_v$ for each vertex $u\in V(R_v)$, we have $|\overline{R_u}| \leq 2$. 
    It follows that
    $\sum_{u \in R_v}|\overline{R_u}| \leq |\overline{R_v}|+2(|V(R_v)|-1)\leq |V(R_v)|-1+2(|V(R_v)|-1)=3|V(R_v)|-3.$
\end{proof}

\begin{lemma}\label{partition}
    Given a vertex $v$ in a $(C_3\text{-}C_3)$-free plane graph $G$ with $R_v\neq\emptyset$, suppose that $(k_1,k_2,...,k_{\ell})$ is a partition of $R_v$. If all the following conditions: 
    \begin{enumerate}
        \item $k_1\leq 2$ when $\ell=3$;
        \item $k_1\leq 3$ and $k_1+k_2\leq 5$ when $\ell=2$;
        \item $k_1\leq 4$ when $\ell=1$;
    \end{enumerate} 
    do not hold, then we have $R_u \subseteq R_v$ for each vertex $u$ in $R_v$.
\end{lemma}

\begin{proof}
    Suppose, to the contrary, that $R_u\not\subseteq R_v$ for some vertex $u$ in $R_v$. 
    It implies that there is an edge $e$ not containing $v$ such that $F=e\cup \{u\}$ is a $3$-face in $G$.

    If $\ell\geq 4$, then we could easily find a copy of $C_3\text{-}C_3$ since the $3$-face $F$ intersects at most three parts of $R_v$, a contradiction.

    If $\ell=3$ and $k_1\geq 3$, recall that $G$ is $(C_3\text{-}C_3)$-free, then three vertices of $F$ belong to three parts of $R_v$, respectively. Without loss of generality, assume that $u$ belongs to the first part. Note that $k_1\geq 3$. It is not difficult to find a copy of $C_3\text{-}C_3$ which is $F$ and a $3$-face of the first part and an edge connecting them, a contradiction.

    If $\ell=2$ and $k_1\geq 4$, recall that $G$ is $(C_3\textbf{-}C_3)$-free, then $F$ contains one vertex of $\{v_1,v_2\}$, one vertex of $\{v_{k_1},v_{k_1+1}\}$ and one vertex of $\{v_{k_1+2},v_{k_1+3}\}$, respectively.
    If $k_1\geq 5$, then we could find a copy of $C_3\textbf{-}C_3$ which is the $3$-face $F$ and the $3$-face $v_3v_4v$ and an edge connecting them, a contradiction.
    Hence, assume that $k_1=4$. If $v_2v_4\in E(F)$, then we could find a copy of $C_3\textbf{-}C_3$ between the triangle $v_2v_3v_4$ and the $3$-face $vv_{k_1+2}v_{k_1+3}$, a contradiction. If $v_1v_4\in E(F)$ or $v_1v_5\in E(F)$ or $v_2v_5\in E(F)$, then it is easy to find a copy of $C_3\textbf{-}C_3$ which is $F$ and the $3$-face $v_2v_3v$ or $v_3v_4v$ and an edge connecting them,  a contradiction.
    
    If $\ell=2$ and $k_1+k_2\geq 6$, from the discussion above, we could assume that $k_1=k_2=3$.
    To avoid the appearance of $C_3\textbf{-}C_3$, the $3$-face $F$ intersects at least one vertex of $\{v_1,v_2\}$, one vertex of $\{v_3,v_4\}$, one vertex of $\{v_5,v_6\}$ and one vertex of $\{v_7,v_8\}$, a contradiction. 

    If $\ell=1$ and $k_1\geq5$, then $R_v$ is a wheel or a fan. First, we assume that $R_v$ is a fan. Since $k_1+1\geq 6$, the $3$-face $F$ must be $v_1v_3v_5$ or $v_2v_4v_6$ otherwise we could find a copy of $C_3\textbf{-}C_3$ which is $F$ and $v_iv_{i+1}v$ and an edge connecting them, where $v_i$ and $v_{i+1}$ are not in $F$. However, a copy of $C_3\textbf{-}C_3$ appears which is the $3$-face $v_5v_6v$ and the triangle $v_1v_2v_3$ or  the $3$-face $v_5v_6v$ and the triangle $v_2v_3v_4$ and an edge connecting them, a contradiction.
    Now we may assume that $R_v$ is a wheel. Since $k_1\geq 5$, the $3$-face $F$ must intersect three vertices with $R_v$ to avoid the appearance of $C_3\textbf{-}C_3$. If $k_1\geq 6$, then the $3$-face $F$ must be $v_1v_3v_5$ or $v_2v_4v_6$ as in the same discussion above. We could also find a copy of $C_3\textbf{-}C_3$, a contradiction. If $k_1=5$, without loss of generality, let $F$ be $v_1v_3v_4$. Then there is a copy of $C_3\textbf{-}C_3$ which is the triangle $v_1v_2v_3$ and the $3$-face $v_4v_5v$ and an edge connecting them, a contradiction.
\end{proof}

Before stating our next lemma, we give the following three special $3$-face-blocks $\mathbf{B}_1, \mathbf{B}_2$ and $\mathbf{B}_3$:

\begin{figure}[H]
    \centering
    \includegraphics[width=0.8\linewidth]{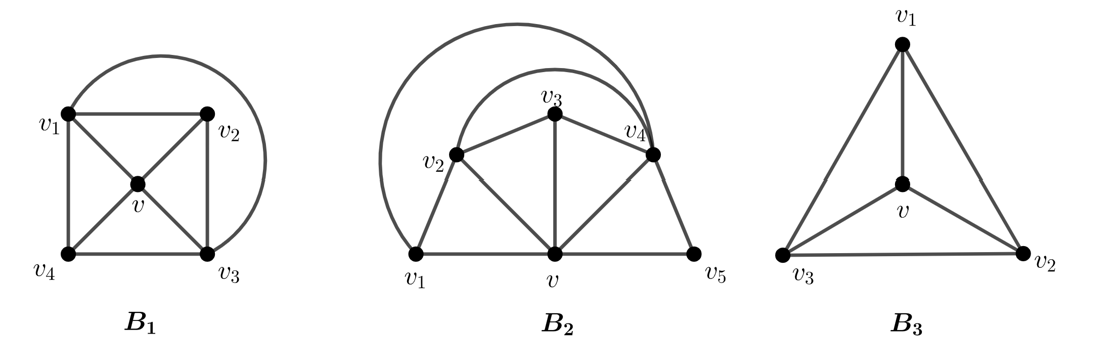}
     \caption{The $3$-face-blocks with $\sum_{v \in B}|\overline{R_v}|\geq3|V(B)|$}
    \label{face-blocks-equation}
\end{figure}

\begin{lemma}\label{face-block}
    For each $3$-face-block $B$ in a $(C_3\textbf{-}C_3)$-free plane graph $G$, we have  
    \begin{enumerate}
        \item $\sum_{v \in B}|\overline{R_v}| \leq 3|V(B)|-3$ when $B\notin\{\mathbf{B}_1,\mathbf{B}_2\}$ or $B\cong\mathbf{B}_3$ and $v_1v_2v_3$ is not a face;
        \item $\sum_{v \in B}|\overline{R_v}| = 3|V(B)|$ when $B\cong\mathbf{B}_2$ or $B\cong\mathbf{B}_1$ and $v_1v_3v_4$ is not a face or $B\cong\mathbf{B}_3$ and $v_1v_2v_3$ is a face;
        \item $\sum_{v \in B}|\overline{R_v}| = 3|V(B)|+3$ when $B\cong\mathbf{B}_1$ and $v_1v_3v_4$ is a face.
    \end{enumerate}
\end{lemma}

\begin{proof}
    It is easy to verify that the statements $2$ and $3$ hold. For each $3$-face-block $B\notin\{\mathbf{B}_1,\mathbf{B}_2,\mathbf{B}_3\}$, we define $|\overline{R_B}|=\mbox{max}\{|\overline{R_u}|:u\in V(B)\}$. From Lemma~\ref{R-v}, if $R_u\subseteq R_v$ for each vertex $u$ of $R_v$, then $R_v=B$ and $\sum_{u \in R_v}|\overline{R_u}| \leq 3|V(R_v)|-3$. 
    Suppose that there is a vertex $u\in V(R_v)$ such that $R_u\not\subseteq R_v$. It follows that there is a $3$-face $F$ containing $u$ such that $F\notin R_v$. From Lemma~\ref{partition}, we have $|\overline{R_B}|\leq 6$. Now we consider the following cases depending on the value of $|\overline{R_B}|$.

    \begin{case} 
        $|\overline{R_B}|=6$.
    \end{case}
    
    From the statement $1$ of Lemma~\ref{partition}, the partition of $R_v$ is $(2,2,2)$. If $v_1=u$, to avoid the appearance of $C_3\textbf{-}C_3$ in $G$, then $F$ must contain $v_5, v_8$ and at least one vertex of $\{v_2,v_3\}$, a contradiction. Similarly, we have $u\notin\{v_1,v_3,v_4,v_6,v_7,v_9\}$. Hence, $u\in\{v_2,v_5,v_8\}$ and $F=v_2v_5v_8$. It implies that $B$ contains at most $(|\overline{R_v}|+1)$ $3$-faces. Thus, $\sum_{v\in B}|\overline{R_v}| \leq 21 < 3\times10-3$. 
    
    \begin{case} 
        $|\overline{R_B}|=5$.
    \end{case}
    
    From the statements $1$ and $2$ of~Lemma \ref{partition}, the partition of $R_v$ is $(2,2,1)$ or $(3,2)$.
    If the former holds, by an argument similar to that above, the $3$-face $F$ is $v_2v_5v_7$ or $v_2v_5v_8$. It follows from the planarity of $G$ that at most one of $v_2v_5v_7$ and $v_2v_5v_8$ is a $3$-face of $G$. Note that $B$ contains at most $(|\overline{R_v}|+1)$ $3$-faces. Thus $\sum_{v \in B}|\overline{R_v}| \leq 18 < 3\times9-3$; 
    If the latter holds, to avoid the appearance of $C_3\textbf{-}C_3$ in $G$, then $v_5v_7, v_1v_3, v_2v_4 \notin E(G)$. Since $G$ is $(C_3\textbf{-}C_3)$-free, $F$ must intersect one vertex of $\{v_5,v_6\}$, $\{v_6,v_7\}$, $\{v_1,v_2\}$ and $\{v_3,v_4\}$, respectively. It follows that $F$ is $v_6v_1v_4$ or $v_6v_2v_3$. If $F$ is $v_6v_1v_4$, then we could find a copy of $C_3\textbf{-}C_3$ which is $F$ and the $3$-face $vv_2v_3$ and an edge connecting them, a contradiction. Then $F$ is $v_6v_2v_3$. By the planarity of $G$, $B$ contains at most $(|\overline{R_v}|+1)$ $3$-faces. Thus, $\sum_{v \in B}|\overline{R_v}| \leq 18 < 3\times8-3$.

    \begin{case} 
        $|\overline{R_B}|=4$.
    \end{case}
    
    From the conditions $1$-$3$ of Lemma~\ref{partition}, the partition of $R_v$ is $(2,1,1)$ or $(2,2)$ or $(3,1)$ or $(4)$.
    
    If the partition of $R_v$ is $(2,1,1)$, by a similar discussion, then $F$ must contain $v_2$, one vertex of $\{v_4,v_5\}$ and one vertex of $\{v_6,v_7\}$, respectively. Thus, $F$ is $v_2v_4v_6$ or $v_2v_4v_7$ or $v_2v_5v_6$ or $v_2v_5v_7$. Since $G$ is planar, $B$ contains at most $(|\overline{R_v}|+1)$ $3$-faces. Then $\sum_{v \in B}|\overline{R_v}| \leq 15 < 3\times8-3$.

    If the partition of $R_v$ is $(3,1)$, to avoid the appearance of $C_3\textbf{-}C_3$ in $G$, then $F$ must contain exact one vertex of $\{v_1,v_2\}$, $\{v_3,v_4\}$ and $\{v_5,v_6\}$, respectively. If $v_1v_3\in E(G)$, then we could find a copy of $C_3\textbf{-}C_3$ which is $F$ and the $3$-face $vv_5v_6$ in $B$ and an edge connecting them, a contradiction. Similarly, we have $v_2v_4, v_1v_4\notin E(G)$. Then $F$ is $v_2v_3v_5$ or $v_2v_3v_6$. But $\{v_2v_3v_5, v_2v_3v_6\}$ cannot exist at the same time by the planarity of $G$. Then $\sum_{v \in B}|\overline{R_v}| \leq 15 < 3\times7-3$.

    If the partition of $R_v$ is $(2,2)$, to avoid the appearance of $C_3\textbf{-}C_3$ in $G$, then $F$ must contain exact one vertex of $\{v_1,v_2\}$, $\{v_2,v_3\}$, $\{v_4,v_5\}$ and $\{v_5,v_6\}$, respectively. By a similar argument, we have $v_1v_3, v_4v_6\notin E(G)$. If $V(F)\subseteq V(R_v)$, then $F$ is $v_2v_4v_5$ or $v_2v_5v_6$ or $v_5v_1v_2$ or $v_5v_2v_3$. Note that there are at most two $3$-faces in $\{v_1v_4v_5, v_1v_4v_6, v_2v_3v_5, v_2v_3v_6\}$ by the planarity of $G$. If $V(F)\not\subseteq V(R_v)$, then $F$ contains $v_2$ and $v_5$. By the maximality of $|\overline{R_B}|$, there are at most two vertices $w_1,w_2\notin V(R_v)$ such that $v_2v_5w_1$ and $v_2v_5w_2$ are $3$-faces. Recall that $G$ is $(C_3\textbf{-}C_3)$-free, $|\overline{R_{w_1}}|=|\overline{R_{w_2}}|=1$(if $w_1$ and $w_2$ exist) and there are at most two choices for $F$ in $\{v_1v_4v_5, v_1v_4v_6, v_2v_3v_5, v_2v_3v_6, v_2v_5w_1, v_2v_5w_2 \}$. Then $\sum_{v \in B}|\overline{R_v}| \leq 18 \leq 3|V(B)|-3$.

    If the partition of $R_v$ is $(4)$, then $R_v$ is a wheel or a fan.
    First, we assume that $R_v$ is a wheel.
    If $F$ contains only one vertex $v_i$ of $R_v$, then we could find a copy of $C_3\textbf{-}C_3$ which is $F$ and the $3$-face $vv_{i+1}v_{i+2}$ and an edge connecting them, where $i\in[4]$ and the subscripts are taken modulo $4$.
    If $F$ contains exact two successive vertices $v_{i}$ and $v_{i+1}$ of $R_v$, then we could find a copy of $C_3\textbf{-}C_3$ which is $F$ and the $3$-face $vv_{i-1}v_{i-2}$ and an edge connecting them, where $i\in[4]$ and the subscripts are taken modulo $4$.
    If $V(F)\subseteq V(R_v)$, then $F$ is $v_1v_2v_3$ or $v_2v_3v_4$ or $v_3v_4v_1$ or $v_4v_1v_2$.
    If $V(F)\not \subseteq V(R_v)$, then $F$ contains $\{v_1, v_3\}$ or $\{v_2, v_4\}$.
    Note that there are at most two vertices $w_1,w_2\notin V(R_v)$ such that $w_1v_1v_{3}$ and $w_2v_1v_3$ or $w_1v_2v_4$ and $w_2v_2v_4$ are $3$-faces.
    By the planarity of $G$, there are at most two choices for $F$ in $\{v_1v_2v_3, v_2v_3v_4, v_3v_4v_1, v_4v_1v_2, w_1v_1v_{3}, w_2v_1v_3, w_1v_2v_4, w_2v_2v_4\}$.
    If both $v_{i}v_{i+1}v_{i+2}$ and $v_{i}v_{i-1}v_{i-2}$ are $3$-faces, then $B\cong\mathbf{B}_1$, a contradiction. 
    If both $w_1v_iv_{i+2}$ and $w_2v_iv_{i+2}$ are $3$-faces, then $|\overline{R_{w_1}}|=|\overline{R_{w_2}}|=1$ by the hypothesis that $G$ is $(C_3\textbf{-}C_3)$-free.
    Thus, $\sum_{v \in B}|\overline{R_v}| = 18 = 3|V(B)|-3$.
    By symmetry, if both $v_{i}v_{i+1}v_{i+2}$ and $w_1v_iv_{i+2}$ are $3$-faces, then $|\overline{R_{w_1}}|=1$ by the hypothesis that $G$ is $(C_3\textbf{-}C_3)$-free.
    Thus, $B\cong\mathbf{B}_2$, a contradiction.
    Now we assume that $R_v$ is a fan. Since $G$ is $(C_3\textbf{-}C_3)$-free, we have $v_1v_3,v_3v_5\notin E(G)$.
    Using a similar argument, we could conclude that $V(F)\subseteq V(R_v)$ or $F$ is the $3$-face $v_2v_4w$, where $w\notin V(R_v)$.
    If $V(F)\subseteq V(R_v)$, then $F$ is $v_2v_3v_4$ or $v_1v_2v_4$ or $v_5v_4v_2$. 
    Note that there are at most two $3$-faces in $\{v_2v_3v_4, v_1v_2v_4, v_5v_4v_2\}$ by the planarity of $G$.
    It follows that there are at most two choices for $F$ in $\{v_2v_3v_4, v_1v_2v_4, v_5v_4v_2, v_2v_4w\}$.
    If both $v_2v_4w$ and $v_2v_3v_4$ are $3$-faces, then $\sum_{v \in B}|\overline{R_v}| \leq 18=3|V(B)|-3$.
    By symmetry of $v_1v_2v_4$ and $v_5v_4v_2$, if both $v_2v_3v_4$ and $v_1v_2v_4$ are $3$-faces, then $B\cong\mathbf{B}_2$, a contradiction. 

    \begin{case}
        $|\overline{R_B}| = 3$. 
    \end{case}

    From the conditions $1$-$3$ of Lemma~\ref{partition}, the partition of $R_v$ is $(1,1,1)$ or $(2,1)$ or $(3)$.
    
    If the partition of $R_v$ is $(1,1,1)$, to avoid the appearance of $C_3\textbf{-}C_3$, then $F$ intersects exact one vertex of $\{v_1,v_2\}$, $\{v_3,v_4\}$ and $\{v_5,v_6\}$, respectively. Note that $B$ contains at most $(|\overline{R_v}|+1)$ $3$-faces. Then $\sum_{v \in B}|\overline{R_v}| \leq 12 < 3|V(B)|-3$.

    If the partition of $R_v$ is $(2,1)$, to avoid the appearance of $C_3\textbf{-}C_3$ in $G$, then $v_1v_3\notin E(G)$ and $F$ intersects at least one vertex of $\{v_1,v_2\}$, $\{v_2,v_3\}$ and $\{v_4,v_5\}$, respectively. 
    We assert that $F$ must contain $v_2$ since $v_1v_3\notin E(G)$, otherwise we could find a copy of $C_3\textbf{-}C_3$ which is the triangle $v_1v_2v_3$ and the $3$-face $vv_4v_5$ and an edge connecting them. 
    Then $F$ must contain $v_2$. By the maximality of $|\overline{R_v}|$ in $B$, there are at most three $3$-faces containing $v_2$. 
    Then $\sum_{v \in B}|\overline{R_v}| \leq 12 < 3|V(B)|-3$.

    If the partition of $R_v$ is $(3)$, then $R_v$ is a wheel or a fan. First, we assume that $R_v$ is a wheel.
    Since $G$ is $(C_3\textbf{-}C_3)$-free, it follows that $F$ intersects at least two vertices of $R_v$.
    If $F$ is $v_1v_2v_3$, then $V(R_v)=V(G)$ and $B\cong\mathbf{B}_3$.
    Hence, $F$ intersects exact two vertices of $R_v$.
    By the maximality of $|\overline{R_B}|$ in $B$, $B$ contains three $3$-faces and four $3$-faces when $|V(B)|=4$ and $|V(B)|=5$, respectively. 
    Then $\sum_{v \in B}|\overline{R_v}| = 3|V(B)|-3$.
    Now we suppose that $R_v$ is a fan. Note that $F$ intersects at least two vertices of $R_v$ by the hypothesis that $G$ is $(C_3\textbf{-}C_3)$-free. 
    If $F$ contains exact two vertices of $R_v$, to avoid the appearance of $C_3\textbf{-}C_3$ in $G$, then $F$ is $wv_1v_3$ or $wv_2v_4$ or $wv_2v_3$, where $w\notin R_v$.
    If $V(F)\subseteq V(R_v)$, then $F$ is $v_1v_2v_3$ or $v_2v_3v_4$ or $v_1v_3v_4$ or $v_1v_2v_4$.
    By the maximality of $|\overline{R_B}|$ in $B$ and the planarity of $G$, $B$ contains at most $(|\overline{R_v}|+1)$ $3$-faces. 
    Then $\sum_{v \in B}|\overline{R_v}| \leq 12 \leq 3|V(B)|-3$.
\end{proof}
 
Now we begin to prove Theorem~\ref{thm1}.

\begin{proof} 
    Suppose that $G$ is connected first.
    It is clear that $ex_{\mathcal{P}}(n,C_3\text{-}C_3)=3n-6$ when $3\leq n\leq 5$.
    If $n=6$, by a similar discussion with Lemma~\ref{face-block}, we could conclude that $G\cong\mathbf{B}_2$ or $G$ has fewer edges than $\mathbf{B}_2$, then we are done.
    Next, we assume that $n\geq 7$ and $e(G) > \lceil 5n/2 \rceil - 5$.
    From Definition~\ref{def}, we could get a set $\mathcal{B}(G)$ of all $3$-face-blocks in the plane graph $G$. 
    Note that all elements in $\mathcal{B}(G)$ are vertex-disjoint and $\mathcal{B}(G)$ is a partition of $V(G)$.
    If each $3$-face-block satisfies $\sum_{v \in B}|\overline{R_v}| \leq 3|V(B)|-3$. 
    Then $$3f_3(G)=\sum_{v \in V(G)}|\overline{R_v}|=\sum_{B \in \mathcal{B}(G)}\sum_{v \in B}|\overline{R_v}|\leq 3n-3,$$ which implies $f_3\leq n-1$.
    
    If there exists a $3$-face-block $B\cong \mathbf{B}_1$ or $B\cong \mathbf{B}_3$, then the outerface of $B$ is not a $3$-face since $n\geq7$, we have $\sum_{v \in B}|\overline{R_v}| \leq 3|V(B)|$ by Lemma~\ref{face-block}.
    Or if there exists a $3$-face-block $B\cong \mathbf{B}_2$, then $\sum_{v \in B}|\overline{R_v}| = 3|V(B)|$ by Lemma~\ref{face-block}. 
    Recall that $G$ is connected, there exists a vertex $u$ such that $N_G(u)\cap V(B)\neq\emptyset$. 
    Note that any two $3$-face-blocks cannot be connected by an edge. Otherwise, we could find a copy of $C_3\textbf{-}C_3$ in $G$.
    Then $|\overline{R_u}|=0$, thus 
    $$3f_3(G)=\sum_{v \in V(G)}|\overline{R_v}|=\sum_{B \in \mathcal{B}(G)}\sum_{v \in B\setminus \{u\}}|\overline{R_v}|+|\overline{R_u}|\leq 3(n-1),$$
    which also implies $f_3(G)\leq n-1$.
    
    It follows that
    \begin{center}
        $2e(G)=\sum_{i\geq3}if_i(G)\geq3f_3(G)+4(f(G)-f_3(G))=4f(G)-f_3(G) \geq 4f(G)-(n-1)$,
    \end{center}
    which implies that $f(G)\leq(2e(G)+n-1)/4$.  By Euler’s formula, we have 
    \begin{center}
        $n-2=e(G)-f(G)\geq e(G)/2-(n-1)/4$.
    \end{center}
    Hence, we have $e(G) \leq \lceil 5n/2 \rceil - 5$ when $n\geq 7$, a contradiction.

    Now suppose that $G$ is not connected. Without loss of generality, assume that $G$ has $t$ connected components $G_1,G_2,...,G_t$. 
    If there is a component $G_i$ with $|V(G_i)|\geq 7$, from the discussion as above, we have $\sum_{v\in V(G_i)}|\overline{R_v}|\leq 3|V(G_i)|-3$. 
    Then $$3f_3=\sum_{v \in V(G)}|\overline{R_v}|=\sum_{B \in \mathcal{B}(G)}\sum_{v \in B}|\overline{R_v}|\leq 3(n-1).$$
    If $|V(G_i)|\leq 6$ for each integer $i\in[t]$, note that $\frac{e(G)}{v(G)}\leq \frac{11}{6}$ see Figure~\ref{face-blocks-equation},
    then $$e(G)=\sum_{i=1}^te(G_i)\leq \mbox{max}\{3n-6t,\frac{11n}{6}\}<\lceil 5n/2 \rceil - 5,$$
    which contradicts the hypothesis that $e(G) > \lceil 5n/2 \rceil - 5$. The result follows.
\end{proof}

We give the extremal graph as follows. Let $P$ be a path with $(n-2)$ vertices and $S$ be a maximum independent set of $P$ containing the two ends of $P$. 
Assume that $G$ is a planar graph with $n$ vertices obtained from $P$ by adding two new adjacent vertices $u$ and $v$ such that $u$ is adjacent to every vertex in $V(P)$ and $v$ is adjacent to every vertex in $S$, respectively. See Figure~\ref{extremal graph} for an example with $n=11$.

\begin{figure}[H]
    \centering
    \includegraphics[width=0.3\linewidth]{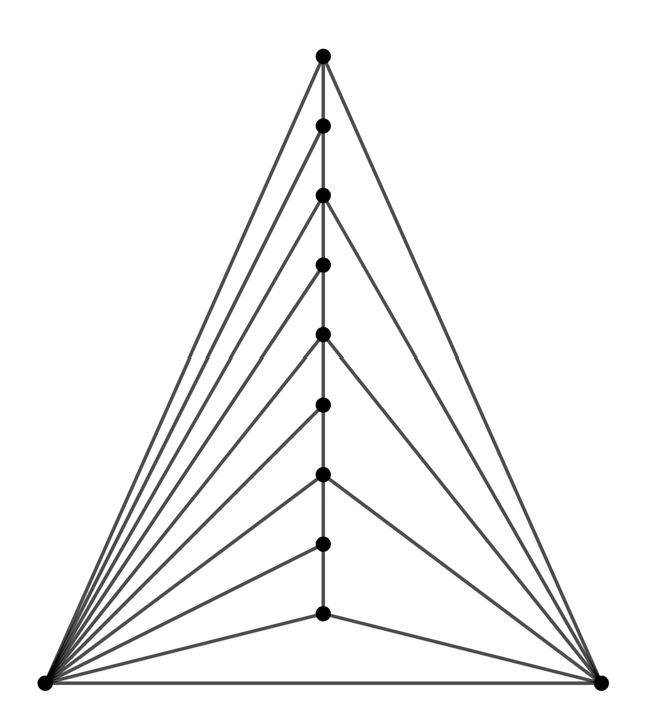}
     \caption{The extremal graph with $n=11$.}
    \label{extremal graph}
\end{figure}

\section{Proof of Theorem~\ref{thm2}}
To prove Theorem~\ref{thm2}, we give some auxiliary lemmas.

\begin{lemma}\label{lemma1}
    For an $n$-vertex plane graph $G$, if there exist $m$ faces are not $3$-faces, then we have $f_3(G) \leq 2n-2m-4$, that is, $\sum_{v \in G}|\overline{R_v}| \leq 6n-6m-12$. In particular, $f_3(G) \neq 2n-5$.
\end{lemma}

\begin{proof}
    Since the maximum number of edges of a planar graph is $3n-6$, and there exist $m$ faces are not $3$-faces, we have $e(G)\leq 3n-m-6$. On the other hand, $2e(G)=\sum_{i=3}^{\infty} i{f_i(G)}\geq 3f_3(G)+4m.$ Combining the two inequalities, we have $f_3(G) \leq 2n-2m-4$. If $m=0$, then $f_3(G) = 2n-4$. If $m \geq 1$, then $f_3(G) \leq 2n-6$. So we always have $f_3(G) \neq 2n-5$.
\end{proof}

\begin{lemma}\label{lemma2}
    For a $3$-face-block $B$ with $\sum_{v \in B}|\overline{R_v}| \leq 3|V(B)|$, after performing the following two operations on $B$ and obtaining a new $3$-face-block $\Hat{B}$, the inequality still holds for $\Hat{B}$.
    \begin{enumerate}
        \item  Adding a new vertex $u$ to form a $3$-face with two vertices in $B$, where $|\overline{R_u}|=1$;
        \item  Adding two new vertices $\{u,w\}$ to form a $3$-face $F$ with a vertex in $B$, and each of $\{u,w\}$ is incident with at most one other $3$-face besides $F$.
    \end{enumerate}
\end{lemma}

\begin{proof}
    For operation $1$: For the new $3$-face-block $\Hat{B}$, since we have $|\overline{R_u}|=1$, then $\sum_{v \in \Hat{B}}|\overline{R_v}| = \sum_{v \in B}|\overline{R_v}|+3 \leq 3|V(B)|+3 = 3|V(\Hat{B})|$.

    For operation $2$: For the new $3$-face-block $\Hat{B}$, since each of $\{u,w\}$ is incident with at most one other $3$-face besides $F$. We have $\sum_{v \in \Hat{B}}|\overline{R_v}| \leq \sum_{v \in B}|\overline{R_v}|+6 \leq 3|V(B)|+6 = 3|V(\Hat{B})|$.
\end{proof}

\begin{lemma}\label{lemma3}
    For a $(C_3\text{-}C_4)$-free plane graph $G$ and a $3$-face-block $B$ of $G$, we have $\sum_{v \in B}|\overline{R_v}| \leq 3|V(B)|$ except the following four cases as shown in Figure~\ref{sepcial-cases}. In particular, there are two properties for the four special $3$-face-blocks.

    \begin{enumerate}
        \item For each $3$-face-block $B$, if it is incident with some vertex $v\notin V(B)$, then we have $|\overline{R_v}|=0$, and the $3$-face-bolck can only have one edge with the vertex $v$;
        \item For each $3$-face-block $B$, it cannot be incident with some $4$-face except the outerface in $B_4$. 
    \end{enumerate}
\end{lemma}

\begin{figure}[H]\label{sepcial-cases}
    \centering
    \includegraphics[width=0.9\linewidth]{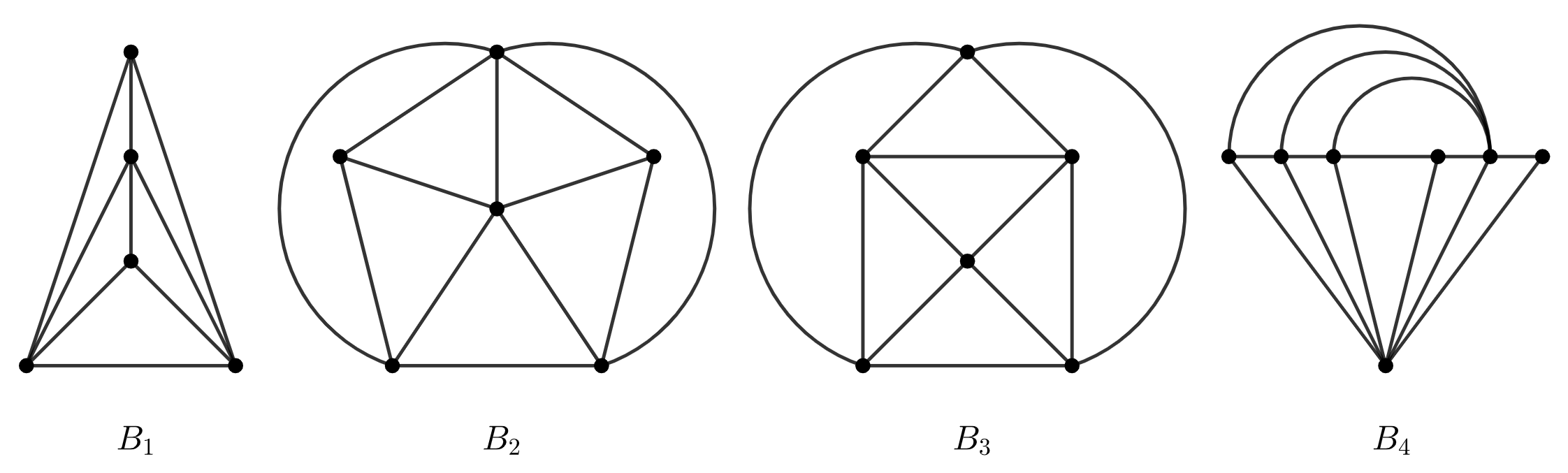}
    \caption{The $3$-face-blocks with $\sum_{v \in B}|\overline{R_v}| > 3|V(B)|$.}
    \label{fig:enter-label}
\end{figure}

\begin{proof}
    
    The two properties are trivial, since the graph $G$ is $(C_3\text{-}C_4)$-free. We give the proof of Lemma~\ref{lemma3} by giving the partition of $R_v$ for some vertex $v\in V(B)$. Let $(k_1,k_2,...,k_\ell)$ denote the partition of $R_v$. We use $t$ to denote the number of $k_i$ which is greater than $1$ for $i=1,2,...,\ell$. We could assume that $|\overline{R_v}|\geq 4$, otherwise we are done since $|\overline{R_v}|\leq 3$ for each $v\in V(B)$. If there exists some vertex $u$ satisfying $R_u \subseteq R_v$, then according to Lemma~\ref{R-v} we have $\sum_{u \in R_v}|\overline{R_u}| \leq 3|V(R_v)|$. So we could always assume that $R_u\not\subseteq R_v$ for some vertex $u$ in $R_v$. It implies that there is an edge $e$ not containing $v$ such that $F=e\cup \{u\}$ is a $3$-face in $G$.
    
    \begin{case}
        $t\geq4$.
    \end{case}
    
    Since $t \geq 4$, we could easily find a copy of $C_3\text{-}C_4$ since the $3$-face $F$ intersects at most three parts of the partition of $R_v$, a contradiction.
    
    \begin{case}
        $t=3$.
    \end{case}
    
    $F$ intersects each part of $1,2,3$ exactly once, there exists only one such $F$ since the graph is planar, we could also conclude that $\sum_{u \in R_v}|\overline{R_u}| \leq 3|V(R_v)|.$
    
    \begin{case}
        $t=2$.
    \end{case}
    
    We start by proving the case where the partition is $(2,2)$. If $F\subseteq R_v$, then we could assume that $v_1 \in F$, since $v_1$ must be incident with some vertices of $\{v_4,v_5,v_6\}$, then there exist at least $2$ faces are not $3$-faces. According to Lemma~\ref{lemma1}, we have $f_3(B)\leq 6<n=7$, which implies $\sum_{u \in R_v}|\overline{R_u}| \leq 3|V(R_v)|$. If $F\not\subseteq R_v$, then there is a vertex $v_0\in F$ but $v_0 \not\in R_v$. If $|\overline{R_{v_0}}|=1$, according to Lemma~\ref{lemma2}, $\sum_{u \in R_v}|\overline{R_u}| \leq 3|V(R_v)|$. If $|\overline{R_{v_0}}|\geq2$, then there are at least two $3$-faces which are incident with $v_0$ and at least $3$ vertices of $\{v_1,v_2,v_3,v_4,v_5,v_6\}$, so we could assume that $v_0$ is incident with two vertices of $\{v_1,v_2,v_3\}$, if the two vertices are $\{v_1,v_3\}$, then there is a copy of $C_3\text{-}C_4$. If the two vertices are $\{v_1,v_2\}$, there is also a copy of $C_3\text{-}C_4$, a contradiction. So for the partition $(2,2)$, we have $\sum_{u \in R_v}|\overline{R_u}| \leq 3|V(R_v)|$. For the partition $(2,2,1,1...)$, the part with only one $3$-face can be incident with at most one more other $3$-face since the graph is planar and $(C_3\text{-}C_4)$-free. Then according to Lemma~\ref{lemma2}, we have $\sum_{u \in R_v}|\overline{R_u}| \leq 3|V(R_v)|$.
    
    For $k_1=3,k_2=2$, we could request that $|\overline{R_{v_1}}| \geq 2$ and $|\overline{R_{v_4}}| \geq 2$. Otherwise, if $|\overline{R_{v_1}}| = 1$, then by Lemma~\ref{lemma2} and the result for the partition $(2,2,1,1,...)$, we have $\sum_{u \in R_v}|\overline{R_u}| \leq 3|V(R_v)|$. Since $|\overline{R_{v_1}}|\geq 2$, another $3$-face incident with $v_1$ must be $\{v_1,v_2,v_6\}$. Similarly, for $v_4$, there is a $3$-face incident with $v_4$ which is $\{v_3,v_4,v_6\}$. In this case, there are at least $2$ faces are not $3$-faces. According to Lemma~\ref{lemma1}, we have $f_3(B)\leq 8$, that is, $\sum_{u \in R_v}|\overline{R_u}| \leq 3|V(R_v)|$. As for $(3,2,1,1,...)$, we could deduce the same inequality as discussed in the former case.
    
    For $k_1=3,k_2=3$, we could request that $\mbox{min}\{\overline{|R_{v_1}}|,|\overline{R_{v_4}}|,|\overline{R_{v_5}}|,|\overline{R_{v_8}}|\} \geq 2$. The other $3$-face incident with $v_1$ must be $\{v_1,v_2\}$ and one vertex of $\{v_6,v_7\}$, no matter which vertex it is, no more $3$-face can be incident with $v_8$, a contradiction.
    
    For $k_1 \geq 4$, we could request that $|\overline{R_{v_1}}| \geq 2$ and $|\overline{R_{v_{k_1+1}}}| \geq 2$. However, no more $3$-face can be incident with $v_1$, since the graph is $(C_3\text{-}C_4)$-free.
    
    Thus, for $t=2$, we have $\sum_{u \in R_v}|\overline{R_u}| \leq 3|V(R_v)|$.
    
    \begin{case}
        $t=1$ and $R_v$ is not a wheel or a fan.
    \end{case}
    
    We start by proving the case where the partition is $(4,1)$. If $F\subseteq R_v$ and $|\overline{R_{v_6}}|=|\overline{R_{v_7}}|=1$, then there are at most three $3$-faces, we have $\sum_{u \in R_v}|\overline{R_u}| \leq 3|V(R_v)|$. If $|\overline{R_{v_6}}| \geq 2$, then $v_6$ must have edges with $\{v_1,v_2,v_3,v_4,v_5\}$, so there are at least $2$ faces are not $3$-faces. According to Lemma~\ref{lemma1}, we have $\sum_{u \in R_v}|\overline{R_u}| \leq 3|V(R_v)|$. If $F\not\subseteq R_v$, then there is a vertex $v_0\in F$ but $v_0 \not\in R_v$. If $|\overline{R_{v_0}}|=1$, according to Lemma~\ref{lemma2}, we have $\sum_{u \in R_v}|\overline{R_u}| \leq 3|V(R_v)|$. If $|\overline{R_{v_0}}|\geq2$, we use $\{F_1,F_2,...,F_p\}$ to denote those $3$-faces which are incident with $v_0$. We claim that each $F_i$ cannot be incident with $\{v_1,v_5\}$ for $i=1,2,...,p$, and $\{v_0,v_2,v_4\}$ cannot be a $3$-face, since the graph is $(C_3\text{-}C_4)$-free. Thus, $|\overline{R_{v_0}}|=2$, without loss of generality, we could assume that $\{v_0,v_2,v_3\}$ and $\{v_0,v_3,v_6\}$ are the two $3$-faces. However, there is a copy of $C_3\text{-}C_4$, which is $\{v,v_4,v_5\}$ and $\{v_0,v_2,v_3,v_6\}$ and the edge $vv_3$. As for the partition $(4,1,1,1...)$ with $m$ single $3$-faces. we could request that for each single $3$-face, it must be incident with two more $3$-face according to Lemma~\ref{lemma2}. If $F \subseteq R_v$, each single $3$-face must have edges with $\{v_1,v_2,v_3,v_4,v_5\}$, then there are at least $m+1$ faces are not $3$-faces. According to Lemma~\ref{lemma1}, we have $\sum_{u \in R_v}|\overline{R_u}| \leq 3|V(R_v)|$, since $n=2m+6$. If $F \not\subseteq R_v$, we have the same result as discussed above. So for $k_1=4$, the inequality holds.
    
    For $k_1=3$, we start by proving the case where the partition is $(3,1)$. If $F\subseteq R_v$ and $\overline{|R_{v_5}}|=|\overline{R_{v_6}}|=1$, then there are at most two $3$-faces, we have $\sum_{u \in R_v}|\overline{R_u}| \leq 3|V(R_v)|$. If $|\overline{R_{v_5}}| \geq 2$, then $v_5$ must have edges with $\{v_1,v_2,v_3,v_4\}$, so there are at least $2$ faces are not $3$-faces. According to Lemma~\ref{lemma1}, we have $\sum_{u \in R_v}|\overline{R_u}| \leq 3|V(R_v)|$. If $F\not\subseteq R_v$, then there is a vertex $v_0\in F$ but $v_0 \not\in R_v$. If $|\overline{R_{v_0}}|=1$, according to Lemma~\ref{lemma2}, we have $\sum_{u \in R_v}|\overline{R_u}| \leq 3|V(R_v)|$. If $|\overline{R_{v_0}}|\geq2$, we use $\{F_1,F_2,...F_p\}$ to denote those $3$-faces which are incident with $v_0$. Then each $F_i$ must contain one of $\{v_2,v_3\}$ for $i=1,2,...,p$, without loss of generality, we could assume that $v_0v_2$ is an edge. Since the graph is $(C_3\text{-}C_4)$-free, we could see that $F_1$ must be $\{v_0,v_2,v_3\}$, $F_2$ must be $v$, one of $\{v_2,v_3\}$ and one of $\{v_5,v_6\}$. By symmetry, we could assume that $F_2$ is $\{v_0,v_2,v_6\}$. Then $\{v_0,v_2,v_5\}$ cannot form a $3$-face, otherwise, $\{v_0,v_5,v_6\}$ , $\{v,v_1,v_2,v_3\}$ and the edge $vv_1$ form a copy of $C_3\text{-}C_4$, a contradiction. So $p=2$, and in this case, $|\overline{R_{v_1}}|=4$ and the maximum number of $|\overline{R_v}|$ is four. Thus, there is no more $3$-face incident with $v_2$, and since $v_1$ cannot be incident with $v_6$, there is also no more $3$-face incident with $v_2$. While $|\overline{R_{v_1}}|=3$, there is at most one $3$-face incident with $v_`$, then there is no more $3$-face incident with any of $\{v,v_1,v_2,v_3\}$, we have $\sum_{u \in R_v}|\overline{R_u}| \leq 3|V(R_v)|$. As for the partition $(3,1,1,1...)$ with $m$ single $3$-faces, we could request that for each single $3$-face, it must be incident with two more $3$-faces according to Lemma~\ref{lemma2}. If $F \subseteq R_v$, then each single $3$-face must have edges with $\{v_1,v_2,v_3,v_4\}$, then there are at least $m+1$ faces are not $3$-faces. According to Lemma~\ref{lemma1}, we have $\sum_{u \in R_v}|\overline{R_u}| \leq 3|V(R_v)|$, since $n=2m+5$. If $F \not\subseteq R_v$, we have the same result as discussed above. So for $k_1=3$, the inequality holds.
    
    For $k_1=2$, recall that $|\overline{R_v}|\geq 4$, we start by proving the case where the partition is $(2,1,1)$. If $F\subseteq R_v$ and $|\overline{R_{v_4}}|=|\overline{R_{v_5}}|=|\overline{R_{v_6}}|=|\overline{R_{v_7}}|=1$, then there are at most two $3$-faces, we have $\sum_{u \in R_v}|\overline{R_u}| \leq 3|V(R_v)|$. If $|\overline{R_{v_4}}|+|\overline{R_{v_5}}|\geq 3$ and $|\overline{R_{v_6}}|=|\overline{R_{v_7}}|=1$, then there are at least two faces are not $3$-faces. According to Lemma~\ref{lemma1}, we have $\sum_{u \in R_v}|\overline{R_u}| \leq 3|V(R_v)|$. If $|\overline{R_{v_4}}|+|\overline{R_{v_5}}|\geq 3$ and $|\overline{R_{v_6}}|+|\overline{R_{v_7}}|\geq 3$, there are at least three faces are not $3$-faces. According to Lemma~\ref{lemma1}, we could get the same result. If $F\not\subseteq R_v$, then there is a vertex $v_0\in F$ but $v_0 \not\in R_v$. If $|\overline{R_{v_0}}|=1$, according to Lemma~\ref{lemma2}, we have $\sum_{u \in R_v}|\overline{R_u}| \leq 3|V(R_v)|$. If $|\overline{R_{v_0}}|\geq2$, we use $\{F_1,F_2,...,F_p\}$ to denote those $3$-faces which are incident with $v_0$. Since each $F_i$ has to be incident with one of $\{v_1,v_2,v_3\}$, and $v_0$ cannot have any edge with $\{v_1,v_3\}$ meanwhile.
    We claim that each $F_i$ must contain $v_2$. Otherwise, we could assume that each $F_i$ contains $v_1$, since $|\overline{R_{v_0}}| \geq 2$. Then $F_1$ and $F_2$ form a $C_4$, together with $\{v,v_2,v_3\}$, they form a $C_3\text{-}C_4$, a contradiction. So, each $F_i$ must contain $v_2$, and $|\overline{R_{v_0}}|=2$. $F_1$ contains $\{v_0,v_2\}$ and one vertex of $\{v_4,v_5\}$, $F_2$ contains $\{v_0,v_2\}$ and one vertex of $\{v_6,v_7\}$. The vertex chosen form $\{v_4,v_5\}$ has been incident with two $3$-faces, since maximum $|\overline{R_v}|=4$, it can be incident with two more $3$-faces. Thus, there is no more $3$-face is incident with any vertex of $\{v,v_0,v_2,v_i\}$, where $v_i$ is the vertex chosen from $\{v_4,v_5\}$, then we have $\sum_{u \in R_v}|\overline{R_u}| \leq 3|V(R_v)|$. As for the partition $(2,1,1,1...)$ with $m$ single $3$-faces, we could request that for each single $3$-faces, it must be incident with two more $3$-faces according to the Lemma~\ref{lemma2}. If $F \subseteq R_v$, each single $3$-face must have edges with $\{v_1,v_2,v_3\}$, then there are at least $m+1$ faces are not $3$-faces. According to Lemma~\ref{lemma2}, we have $\sum_{u \in R_v}|\overline{R_u}| \leq 3|V(R_v)|$, since $n=2m+4$. If $F \not\subseteq R_v$, we have the same result as discussed above. So for $k_1=2$, the inequality holds.
    
    For $k_1\geq5$, since we have proved for the case when $k_1=4$, we have $\sum_{u \in R_v}|\overline{R_u}| \leq 3|V(R_v)|$, then according to Lemma~\ref{lemma3}, we could request that $|\overline{R_{v_1}}|\geq 2$ and $|\overline{R_{v_{k_1+1}}}|\geq 2$. However, since the graph is $(C_3\text{-}C_4)$-free, $\{v_1,v_3\}$ and $\{v_1,v_4\}$ cannot be an edge. So another $3$-face incident with $v_1$ must be $\{v_1,v_2,v_5\}$, however, there is still a copy of $C_3\text{-}C_4$, which is formed by $\{v_2,v_3,v_4,v_5\}$ and the single $3$-face, a contradiction.
    
    Then for $t=1$, we have $\sum_{u \in R_v}|\overline{R_u}| \leq 3|V(R_v)|$.
    
    \begin{case}
        $t=0$.
    \end{case}
    
    Assume that $|\overline{R_v}| = m \geq 4$, there are at least four single $3$-faces. If there exist two consecutive single $3$-faces have edges, we could relabel them as Figure~\ref{relabel} shows, and prove in the same way as proving the case when the partition is $(3,1,1,......)$. Thus, we have $\sum_{u \in R_v}|\overline{R_u}| \leq 3|V(R_v)|$.
    
    \begin{figure}[H]
        \centering
        \includegraphics[width=0.9\linewidth]{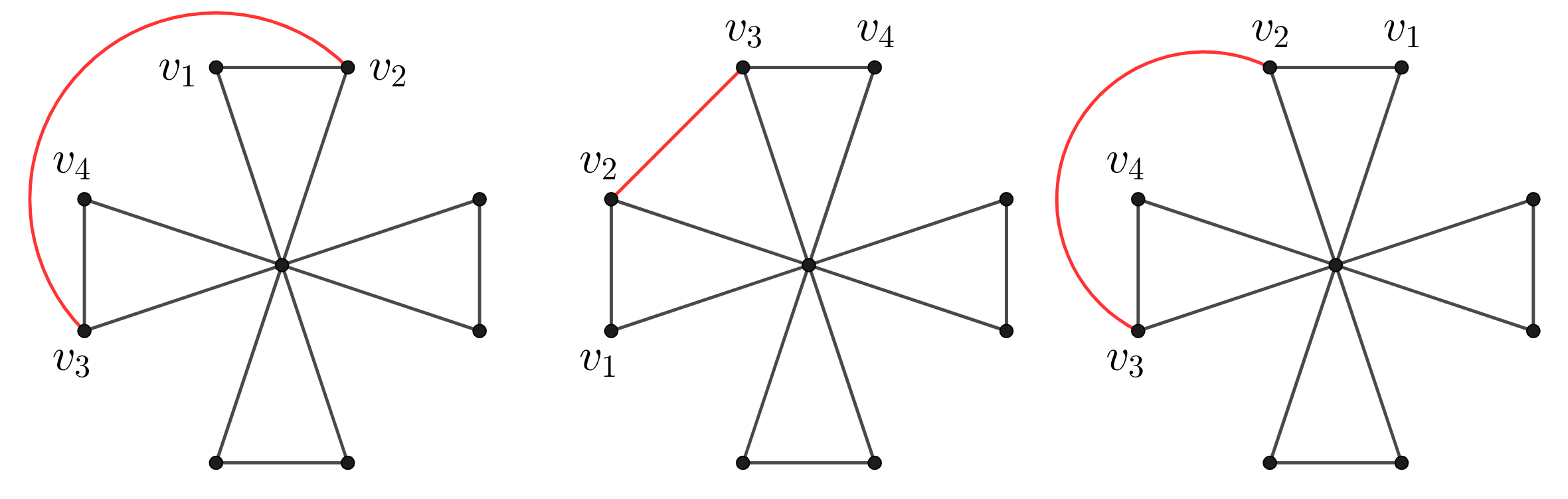}
        \caption{Relabel the vertices.}
        \label{relabel}
    \end{figure}
    
    Then we could assume that any two consecutive single $3$-faces have no edge between them. If there exists a $3$-face-block generated from the partition $(1,1,1,1,...)$ which has the property that $\sum_{u \in R_v}|\overline{R_u}| > 3|V(R_v)|$, then we choose the graph that has the least vertices. Then according to Lemma~\ref{lemma2}, each vertex in this $3$-face-block has the property that $|\overline{R_v}| \geq 2$, otherwise we could find a graph with fewer vertices which satisfies $\sum_{u \in R_v}|\overline{R_u}| > 3|V(R_v)|$. Recall the definition of $3$-face-block, let $w_0$ be a vertex that is incident with $(1,1,1,1,...)$. Let us consider the graph $R_v \cup w_0$. We claim that $w_0$ cannot have edges with $\{v_{2i-1},v_{2i}\}$ at the same time, otherwise $\{v_1,v_2...v_{2m}\}$ must have edges with $w_0$, since $|\overline{R_v}| \geq 2$. Then for the vertices incident with $R_v \cup w_0$, they must have edges with $w_0$, and $3$-faces incident with $w_1$ cannot contain any vertex in $R_v \cup w_0$ since the graph is $(C_3\text{-}C_4)$-free. Moerover, we could claim that there is no $\Theta_4$ in the graph, otherwise there must have triangle to conncect the $\Theta_4$ to the partition $(1,1,1,1,...)$. Hence we have $e_{3,3}=0$ and then $3f_3(G) = e_3$. According to Property~\ref{pro1} we have $3f_3(G) = e_3 \leq e$.
    
    Then for $t=0$, we have $\sum_{u \in R_v}|\overline{R_u}| \leq 3|V(R_v)|$.
    
    \begin{case}
         $t=1$ and $R_v$ is a fan.
    \end{case}
    
    When $|\overline{R_v}| = 6$, If $F\subseteq R_v$, we could see that $|\overline{R_{v_1}}|=|\overline{R_{v_7}}|=1$, since each of the pairs $\{v_1v_3\},\{v_1v_4\},\{v_2v_5\}$ cannot be an edge. Similarly, we could see that $|\overline{R_{v_2}}|=|\overline{R_{v_6}}|=1$, since each of the pairs $\{v_2v_4\},\{v_2v_5\},\{v_3v_6\}$ cannot be an edge. Thus, if $F\subseteq R_v$, there is at most one extra $3$-face: $\{v_3,v_4,v_5\}$, we have $\sum_{u \in R_v}|\overline{R_u}| \leq 3|V(R_v)|$. If $F\not\subseteq R_v$, then there is a vertex $v_0\in F$ but $v_0 \not\in R_v$. If $|\overline{R_{v_0}}|=1$, according to Lemma~\ref{lemma2}, $\sum_{u \in R_v}|\overline{R_u}| \leq 3|V(R_v)|$. If $|\overline{R_{v_0}}|\geq2$, with the same discussion, we could claim that $F_1=\{v,v_3,v_4\}$, $F_2=\{v,v_4,v_5\}$. However, we could easily find a copy of $C_3\text{-}C_4$, a contradiction. So we have $\sum_{u \in R_v}|\overline{R_u}| \leq 3|V(R_v)|$.
    
    For $|\overline{R_v}|\geq7$, we could easily check that $|\overline{R_{v_1}}|=|\overline{R_{v_{k_1+1}}}|=1$. According to Lemma~\ref{lemma2}, we have $\sum_{u \in R_v}|\overline{R_u}| \leq 3|V(R_v)|$.
    
    When $|\overline{R_v}|=5$, if $F\subseteq R_v$ and $|\overline{R_{v_1}}|=|\overline{R_{v_6}}|=1$, there are at most two $3$-faces which use vertices $\{v_2,v_3,v_4,v_5\}$, then we have $\sum_{u \in R_v}|\overline{R_u}| \leq 3|V(R_v)|$. If $F\subseteq R_v$ and $|\overline{R_{v_1}}|+|\overline{R_{v_6}}| \geq 3$, without loss of generality, we could assume that $|\overline{R_{v_1}}| \geq 2$. Since each of the pairs $\{v_1v_3\}$, $\{v_1v_4\}$ cannot be an edge, we could see that the only extra $3$-face incident with $v_1$ is $\{v_1,v_2,v_5\}$. Then if $\{v_3v_5\}$ is not an edge, we have $\sum_{u \in R_v}|\overline{R_u}| \leq 3|V(R_v)|$, if $\{v_3 v_5\}$ is an edge, we have $G \cong B_4$. By the way, we could easily check that $B_4$ cannot be incident with the vertex $u$ while $|\overline{R_u}| \geq 1$. If $F\not\subseteq R_v$, then there is a vertex $v_0\in F$ but $v_0 \not\in R_v$. If $|\overline{R_{v_0}}|=1$, according to Lemma~\ref{lemma2}, $\sum_{u \in R_v}|\overline{R_u}| \leq 3|V(R_v)|$. However, if $|\overline{R_{v_0}}| \geq 2$, we could see that $v$ can only be incident with $\{v_2,v_3,v_4,v_5\}$, since the graph is $(C_3\text{-}C_4)$-free, $F$ cannot be $\{v_0,v_2,v_3\} , \{v_0,v_2,v_4\}$, So $F_1=\{v_0,v_2,v_5\}$, Then $\{v_0,v_3,v_4\}$ cannot be a $3$-face, otherwise there is a copy of $C_3\text{-}C_4$, thus $|\overline{R_{v_0}}|=1$, a contradiction.
    
    When $|\overline{R_v}|=4$, if $F\subseteq R_v$ and $f_3(G) > n = 6$, according to the Lemma~\ref{lemma1}, $f_3(G)\neq 7$, so we have $f_3(G)=8$, which means that all the faces in the $3$-face-block are $3$-faces. However, if all faces in the $3$-face block are $3$-faces, the vertex $v$ must be incident with one more $3$-face, contrary to the assumption that the maximum number of $|\overline{R_v}|$ is $4$. So if $F\subseteq R_v$, we have $f_3(G) \leq n$, that is, $\sum_{u \in R_v}\overline{|R_u|} \leq 3|V(R_v)|$. If $F\not\subseteq R_v$, then there is a vertex $v_0\in F$ but $v_0 \not\in R_v$. If $|\overline{R_{v_0}}|=1$, according to Lemma~\ref{lemma2}, $\sum_{u \in R_v}|\overline{R_u}| \leq 3|V(R_v)|$. If $|\overline{R_{v_0}}|\geq2$, we claim that $v_0$ must have an edge with $v_3$. Otherwise, we could see that $v_0$ cannot have edges with $v_1$ and $v_2$ at the same time, and cannot have edges with $v_4$ and $v_5$ at the same time, which contradicts $|\overline{R_{v_0}}|\geq2$. It is easy to check that $F_1=\{v,v_2,v_3\}, F_2=\{v,v_3,v_4\}$, then there is no more extra $3$-face, so we also have $\sum_{u \in R_v}|\overline{R_u}| \leq 3|V(R_v)|$.
    
    Thus, for the fan case, we have $\sum_{u \in R_v}|\overline{R_u}| \leq 3|V(R_v)|$ or $G \cong B_4$.
    
    \begin{case}
        $t=1$ and $R_v$ is a wheel.
    \end{case}
    
    When $|\overline{R_v}| \geq 6$, for any two vertices $v_i$ and $v_j$, meanwhile, we assume that $i>j$. If $i-j \equiv 2\ \text{or}\  3 \pmod{k_1+1}$, then $v_i$ and $v_j$ cannot form an edge since the graph is $(C_3\text{-}C_4)$-free.
    
    When $|\overline{R_v}|=5$, if $F\subseteq R_v$, we have $G \cong B_2$. If $F\not\subseteq R_v$, we could easily find a copy of $C_3\text{-}C_4$ in the graph.
    
    When $|\overline{R_v}|=4$, if $F\subseteq R_v$, we have $G \cong B_1$. If $F\not\subseteq R_v$, we have the following two cases: the induced subgraph of $v$ and its neighbors is $B_1$ or $W_4$.
    
    For the former case, we label the vertices as Figure~\ref{relabel2} shows.
    
    \begin{figure}[H]
        \centering
        \includegraphics[width=0.4\linewidth]{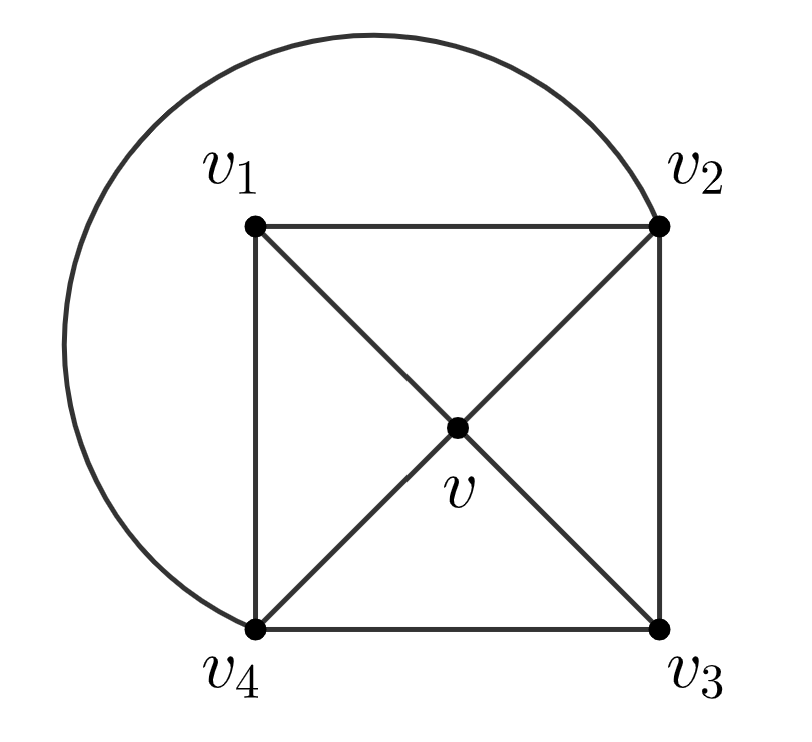}
        \caption{Label the graph of $B_1$.}
        \label{relabel2}
    \end{figure}
    
    If $|\overline{R_{v_0}}|=1$, according to Lemma~\ref{lemma2}, $\sum_{u \in R_v}|\overline{R_v}| \leq 3|V(R_v)|$. If $|\overline{R_{v_0}}|\geq2$, without loss of generality, we could assume that $F_1=\{v_0,v_2,v_3\}$ and $F_2=\{v_0,v_3,v_4\}$, then we have $G \cong B_2$.
    
    For the latter case. If $|\overline{R_{v_0}}|=1$, according to Lemma~\ref{lemma2}, $\sum_{u \in R_v}|\overline{R_u}| \leq 3|V(R_v)|$. If $|\overline{R_{v_0}}|=3$, then we have $G \cong B_3$. If $|\overline{R_{v_0}}|=2$, then there are $6$ vertices and six $3$-faces, so if we have other vertices $w \in R_v$, we could request that $|\overline{R_w}|\geq2$, however, in this case, we could easily find a copy of $C_3\text{-}C_4$. 
    
    Then for the wheel case, we have $\sum_{u \in R_v}|\overline{R_u}| \leq 3|V(R_v)|$ or the graph $G \cong \{B_1,B_2.B_3\}$.
    
\end{proof}

Now we begin to prove Theorem~\ref{thm2}.

We use $|R_v^0|$ to denote the number of vertices which satisfy $|\overline{R_v}|=0$. A $3$-face-block is called a \emph{good-block} if $\sum_{u \in R_v}|\overline{R_u}| \leq 3|V(R_v)|$, note that a vertex $v$ with $|\overline{R_v}|=0$ is also a good-block. Otherwise, we call the block as \emph{bad-block}. As we discussed above, there are only four bad-blocks, which are $\{ B_1,B_2,B_3,B_4\}$. Since the graph is connected, and according to the Lemma~\ref{lemma3}, each bad-block should be incident with a vertex $v_0$ with $|\overline{R_{v_0}}|=0$. Since $B_1$ has $5$ vertices and six $3$-faces while the outerface is a $3$-face, when $B_1$ is incident with a vertex $v_0$ with $|\overline{R_{v_0}}|=0$, the block will reduce one $3$-face, then the block satisfies $\sum_{u \in R_v}|\overline{R_u}| \leq 3|V(R_v)|$. So when the graph is connected, bad blocks only can be one of $\{ B_2,B_3,B_4\}$.

We assume that the number of copies of $B_2$ and $B_3$ in $G$ is $k_1$, the number of copies of $B_4$ in $G$ is $k_2$, and we assume that $|R_v^0|=m$. Since bad-block has the property that $f_3(B)=V(B)+1$, if $m \geq k_1+k_2$, then for the graph we have $f_3(G) \leq n$.

It follows that:

$$2e(G)=\sum_{i\geq3}if_i(G)\geq3f_3(G)+4(f(G)-f_3(G))=4f(G)-f_3(G)\geq4f(G)-n,$$

which implies that $f(G)\leq(2e(G)+n)/4$. By Euler’s formula, we have 

$$n-2=e(G)-f(G)\geq e(G)/2-n/4.$$

Hence, we have $e(G) \leq \lfloor 5n/2 \rfloor - 4$ when $n\geq8$.

So, we only need to discuss the case when $m < k_1+k_2$. We use $G_1$ to denote all good-blocks except single vertices. We assume that $|V(G_1)|=n_1$, let $G_2$ denote all the vertices that satisfy $|\overline{R_v}|=0$, let $G_3$ denote all the bad-blocks. According to Lemma~\ref{lemma3}, there is no edge between $G_1$ and $G_3$. As all the good-block, we have $\sum_{u \in R_v}|\overline{R_u}| \leq 3|V(R_v)|$, since all the vertices in $G_2$ have the property that $|\overline{R_v}|=0$, then for $G_1$ and $G_2$, we have $\sum_{u \in V(G_1 \cup G_2)}|\overline{R_u}| \leq 3|V(G_1)|$. So the total number of $3$-faces in $G_1$ and $G_2$ is at most $n_1$, then we have the following inequality.

$$e(G_1)+e(G_2)+e(G_1:G_2) \leq 2(m+n_1)-4+n_1/2.$$

If we regard each bad-block as a vertex, then according to Lemma~\ref{lemma3}, the vertices represent bad-blocks and $G_2$ form a bipartite graph $H$, and there is no $4$-cycle in the bipartite graph. So, all the faces in this bipartite graph are of size at least $6$. If we assume that there are $\Hat{e}$ edges in the bipartite graph, then we have

$$2\Hat{e}=\sum_{i\geq6}if_i(H) \geq 6f_6(H)+7(f-f_6(H))=7f(H)-f_6(H) \geq 6f(H),$$

which implies that $f \leq \Hat{e}/3$. By Euler’s formula, we have

$$v-2=\Hat{e}-f \geq 2\Hat{e}/3,$$

then we have

$$\Hat{e} \leq 3v/2-3,$$

since $v=m+k_1+k_2$, we have

$$e(G_3)+e(G_2:G_3) \leq 12k_1+14k_2+3(m+k_1+k_2)/2-3.$$

Then for the graph $G$, we have $|V(G)|=n_1+m+6k_1+7k_2$ and $e(G)=e(G_1)+e(G_2)+e(G_3)+e(G_1:G_2)+e(G_2:G_3)$. Then we have the following inequality:

$$E(G) \leq 2(m+n_1)-4+n_1/2+12k_1+14k_2+3(m+k_1+k_2)/2-3,$$

that is

$$E(G) \leq 5n_1/2+7m/2+27k_1/2+31k_2/2-7.$$

Since we have the assumption that $m < k_1+k_2$, we have

$$E(G) \leq 5n_1/2+5m/2+29k_1/2+33k_2/2-7 \leq \lfloor 5(n_1+m+6k_1+7k_2)/2 \rfloor - 4.$$

As for $n \leq 7$, we have discussed in Lemma~\ref{lemma3}.

If the graph is disconnected and has $r$ connected components, which are denoted as $\{H_1,H_2,...H_r\}$, $r\geq 2$. When $|V(H_i)| \geq 8$, we have $E(H_i) \leq \lfloor 5|V(H_i)|/2 \rfloor - 4$. When $|V(H_i)| \leq 7$, we have $E(H_i) \leq 2|V(H_i)|$. Then we have

$$E(G) \leq \sum_{i=1}^{r} \mbox{max}\{\lfloor 5|V(H_i)|/2 \rfloor - 4,2|V(H_i)|\}<\lfloor 5|V(G)|/2 \rfloor - 4.$$

We give the extremal graph as follows. The extremal graph is generated from the partition $(4,1,1,1...)$. Let $v_3$ have edges with all other vertices. If we assume that there are $t$ single $3$-faces, $v=2t+6, e=5t=11$, then we could easily check that there is no $C_3\text{-}C_4$ in $G$. See Figure~\ref{extremal graph 2} for an example with $n=10$.

\begin{figure}[H]
    \centering
    \includegraphics[width=0.6\linewidth]{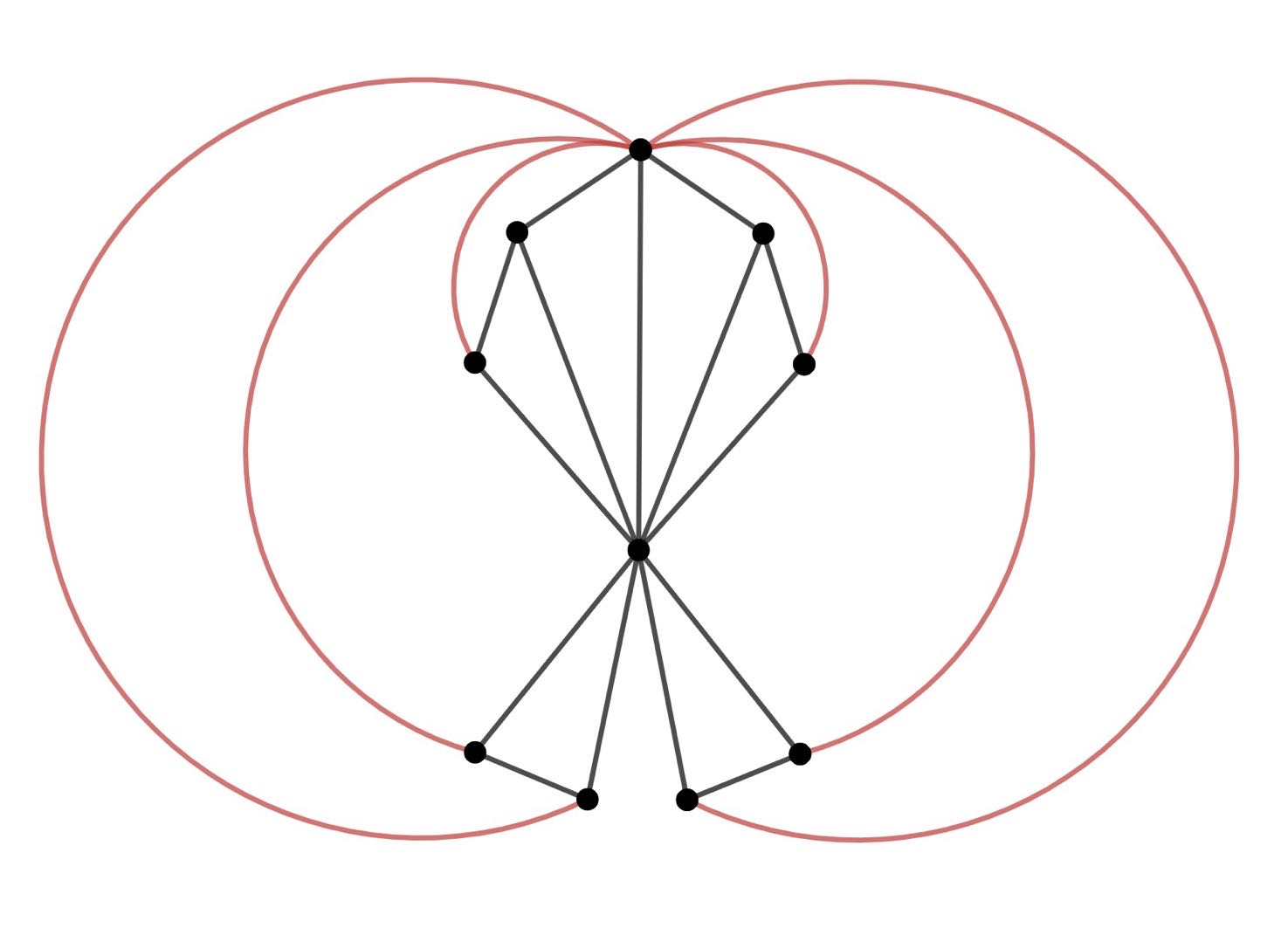}
    \caption{The extremal graph with $n=10$.}
    \label{extremal graph 2}
\end{figure}

\section{Acknowledgments}
The authors would like to thank the anonymous referees for their valuable comments, which greatly improved the presentation of the results. The authors would also like to thank Luyi Li and Tong Li for helpful discussions. This work was supported by the National Key R\&D Program of China (No. 2023YFA1009602) and the National Natural Science Foundation of China (Grant No. 12231018).

\bibliographystyle{abbrv}
\bibliography{main}
\end{document}